\documentclass[preprint, 12pt]{elsarticle}
\usepackage{amssymb}
\usepackage{amsthm}
\usepackage{amsfonts}
\usepackage{amsmath}
\usepackage{algorithm}
\usepackage{algpseudocode}
\usepackage{mathtools}
\usepackage{natbib}
\usepackage{color}
\newtheorem{theorem}{Theorem}[section]
\newtheorem*{theorem*}{Theorem}

\newtheorem{example}[theorem]{Example}
\newtheorem{remark}[theorem]{Remark}
\newtheorem{problem}[theorem]{Problem Statement}

\begin{document}

\begin{frontmatter}
\title{ Solution of the Linearly Structured   Partial Polynomial Inverse
Eigenvalue Problem }
\author[rvt]{Suman Rakshit} %\fnref{fn1}}
\ead{sumanrakshit1991@gmail.com}
\author[rvt]{S. R. Khare }%\corref{cor2}}%\fnref{fn2}}
\ead{srkhare@maths.iitkgp.ernet.in}
 %\cortext[cor2]{
%corresponding author}
%\corref{cor2}

\address[rvt]{Department of Mathematics, Indian Institute of Technology Kharagpur, Kharagpur 721302, India}

\begin{abstract}
In this paper, linearly structured partial polynomial inverse
eigenvalue problem  is considered  for the $n\times n$  matrix
polynomial of arbitrary degree $k$. Given a set of $m$
 eigenpairs ($1 \leqslant m \leqslant kn$), this problem
concerns with computing the matrices $A_i\in{\mathbb{R}^{n\times
n}}$ for $i=0,1,2, \ldots ,(k-1)$ of specified linear structure
such that the matrix polynomial $P(\lambda)=\lambda^k I_n
+\sum_{i=0}^{k-1} \lambda^{i}
 A_{i}$ has the given  eigenpairs as its eigenvalues and eigenvectors.
Many practical applications give rise to the linearly structured
structured matrix polynomial.
  Therefore, construction of the linearly structured  matrix polynomial  is the most important aspect of the polynomial inverse
eigenvalue problem(PIEP). In this paper, a necessary and
sufficient condition for the existence of the solution of this
problem is derived. Additionally, we characterize the class of all
solutions to this problem by giving the explicit expressions of
solutions. The results presented in this paper address some
important open problems in the area of PIEP raised in De Teran,
Dopico and Van Dooren [\emph{SIAM Journal on Matrix Analysis and
Applications}, $36(1)$ ($2015$), pp $302-328$].  An attractive
feature of our solution approach is that it does not impose any
restriction  on the number of eigendata for computing the solution
of PIEP. The proposed method is validated with various numerical
examples on a spring mass problem.
\end{abstract}

\end{frontmatter}

\section{Introduction}

Consider the higher order system of ordinary differential
equations of the form
\begin{equation}
\label{eqn:system_poly} A_k \frac{d^k v(t)}{dt^k}+A_{k-1}
\frac{d^{k-1} v(t)}{dt^{k-1}}+\cdots + A_1 \frac{d v(t)}{dt}+ A_0
v(t) =0
\end{equation}
where $A_i\in{\mathbb{R}^{n\times n}}$ for  $i=0,1,2, \ldots  ,k$
and $A_k$ is a nonsingular matrix.

Assuming  the solution of \eqref{eqn:system_poly} is  of the form
$v(t)=x e^{\lambda t}$, using separation of variables,
\eqref{eqn:system_poly} leads to the higher order polynomial
eigenvalue problem
\begin{equation}
\label{eqn:system_polyEIG}
 P(\lambda)x=0
\end{equation}
where $P(\lambda)= \lambda^k A_k + \lambda^{k-1}
A_{k-1}+\cdots+\lambda A_1+A_0\in{\mathbb{R}^{n\times
n}[\lambda]}$ is known as matrix polynomial of degree $k$. The
comprehensive theory and application of the matrix polynomial  is
discussed in the classic reference \cite{gohberg2005matrix}.

A matrix polynomial $P(\lambda)$ is regular when $P(\lambda)$ is
square and the scalar polynomial det($P(\lambda))$ has at least
one nonzero coefficient. Otherwise, $P(\lambda)$ is said to be
singular. We assume the matrix polynomial $P(\lambda)$ is regular
throughout this paper.
 The  roots of  $det(P(\lambda))=0$ are the
eigenvalues of the matrix polynomial $P(\lambda)$. The vectors $y
\neq 0$ and $z \neq 0$ are corresponding left and right
eigenvectors satisfying
 $z^H P(\lambda)=0$ and  $P(\lambda)y=0$ where $z^H$ denotes the conjugate transpose of $z$.
If the matrix  $A_k$ is nonsingular, then the matrix polynomial
$P(\lambda)$   has  $kn$ finite eigenvalues and eigenvectors.  The
$kn$ eigenvalues of $P(\lambda)$ are either real numbers or if
not, are  complex conjugate pairs.
 The polynomial eigenvalue problem concerns with
determining the eigenvalues and corresponding eigenvectors of the
 matrix polynomial $P(\lambda)$. This problem arises in many
practical situations, for instance, vibration analysis of
structural mechanical and acoustic system, electrical circuit
simulation, fluid mechanics, etc
\cite{datta2010numerical,friswell1995finite}. This  problem is
well studied  in the literature and a lot of literature exists
addressing the ways to solve the polynomial problem (see
\cite{berhanu2005polynomial,datta2010numerical,mackey2006structured,
tisseur2001quadratic} and the references therein).

Mostly, matrix polynomial arising from practical applications are
often inherently structured.
 For
example, they are all symmetric \cite{cai2009solutions},
skew-symmetric \cite{dmytryshyn2014skew}, they alternate between
symmetric and skew-symmetric \cite{mehrmann2002polynomial},
symmetric tridiagonal \cite{bai2007symmetric}, etc. Also,
pentadiagonal matrices occur in the discretization of the
fourth-order differential systems \cite{gladwell1986inverse}.
Generally, these matrices $A_i$ for $i=0,1,2, \ldots ,k$ are
linearly structured matrices \cite{lancaster2002lambda}. A matrix
polynomial $P(\lambda) $ in which  the coefficient matrices are
linearly structured, is known as linearly structured matrix
polynomial.
 Since the matrix $A_k$ is often
diagonal and positive definite in various applications, we assume,
without loss of generality, that the leading coefficient $A_k$ is
an identity matrix. In this case, the matrix polynomial is
referred to as a monic matrix polynomial of degree $k$.

The polynomial  inverse eigenvalue problem (PIEP) addresses the
construction of a matrix polynomial  $P(\lambda)=\sum_{i=0}^{k}
\lambda^{i} A_{i} \in{\mathbb{R}^{n\times n}[\lambda]}$ from the
given eigenvalues and associated eigenvectors. PIEP arises in many
applications where parameters of a certain physical system are to
be determined from the knowledge  of its dynamical behavior. It
has  applications in  the mechanical vibrations, aerospace
engineering, molecular spectroscopy, particle physics, geophysical
applications, numerical analysis, differential equations etc (see
for instance
\cite{barcilon1974solution,barnes1995inverse,chu1998inverse,chu2001inverse,muller1992inverse,parker1981numerical}).

Generally,  a small number of eigenvalues and eigenvectors  of the
associated eigenvalue problem are available from the computation
or measurement. Unfortunately there is no analytical  tool
available to evaluate the entire eigendata of a large  physical
system. It should be mentioned that when the problem is large, as
in the case with the most engineering applications, state of art
computational methods are capable of computing a very few
eigenvalues and associated eigenvectors. Therefore, it might be
more sensible to solve the polynomial inverse eigenvalue problem
when only a few measured eigenvalues and associated eigenvectors
are available.

 The
construction of the matrix polynomial $P(\lambda)=\sum_{i=0}^{k}
\lambda^{i}
 A_{i} \in{\mathbb{R}^{n\times n}[\lambda]}$  using the partially described eigendata is  known as the partial polynomial inverse
eigenvalue problem(PPIEP). In  view of practical applications, it
might be more realistic to solve PPIEP with these structure
constraints on the coefficient matrices. This problem is termed as
the structured partial polynomial inverse eigenvalue
problem(LPPIEP). The structure constraint imposes a great
challenge for solving this problem.

The inverse eigenvalue  problem (IEP) for linear and quadratic
matrix polynomial have been well studied in the  literature  since
the $1970$s  (see \cite{de1979imbedding} the references therein).
Some previous attempts at solving the inverse eigenvalue  problem
are listed in \cite{al2009inverse,
gladwell2014test,hald1976inverse,parlett2016inverse,pickmann2007inverse}.
 A large number of papers have been published on
the linear inverse eigenvalue  problem
\cite{elhay2002affine,rakshit2019symmetric,ram1993inverse}. An
excellent review of this area can be found in the classic
reference \cite{chu1998inverse}. Special attention is paid to the
quadratic inverse eigenvalue problem(QIEP)
(see\cite{bai2007symmetric,cai2009solutions,chu2004inverse,datta2001theory,datta2011solution,
kuo2006solutions,lancaster2007inverse,lancaster2014inverse,ram1996inverse,yuan2011class}).
Most of the papers solve QIEP for the symmetric structure
(see\cite{yuan2011class,cai2009solutions,kuo2006solutions}) and
symmetric tridiagonal structure
(see\cite{bai2007symmetric,ram1996inverse}). The quadratic inverse
eigenvalue problem is considered in the context of solving the
finite element model updating problem
\cite{friswell1995finite,moreno2009symmetry,
 mottershead1993model,
suman2017fem} and eigenstructure  assignment problem
\cite{datta2000partial,nichols2001robust}.

Some earlier  attempts at solving the higher order PIEP are listed
in
\cite{barcilon1974solution,batzke2014inverse,de2015matrix,mclaughlin1976inverse}.
Also, IEP for the matrix polynomial of degree $k $ is considered
in the context of solutions of active vibration control
(see\cite{cai2012robust,mao2013minimum,ramadan2010partial,wang2013partial}).
Most significant  contributions to the solution of the higher
order PIEP have been made in
\cite{batzke2014inverse,de2015matrix}.
 In
\cite{batzke2014inverse}, higher order PIEP for the
$T$-Alternating and $T$-Palindromic  matrix polynomials of degree
$k$ are considered. These results are most phenomenal so far on
the solution of higher order structured PIEP.
 In \cite{de2015matrix}, authors mention an important open
problem in this area, namely, the inverse eigenvalue problems for
structured matrix polynomials such as symmetric, skew-symmetric
matrix polynomials, etc. In this paper, we attempt at addressing
this open problem providing the solution of PIEP.

 Throughout this paper, we shall
adopt the following notations.
  $A \otimes B$ denotes the
Kronecker product of the two matrices $A$ and
 $B$. Also, $\mathrm{ Vec}(A)$ denotes the vectorization of the matrix
 $A$. $\Vert A  \Vert _F$ and $\Vert A  \Vert _2$ denote the Frobenius norm and $2$-norm of the
 matrix $A$ respectively.  $\mathcal{L}$  denotes the real
linear subspaces of $\mathbb{R}^{n\times n}$ representing the
linearly structured  matrices. $A^\dag$ is the Moore Penrose
pseudoinverse of $A$. $I_n $ denotes the identity matrix of size
$n\times n $. Also, $e_i$ is the $i^{\mathrm{th}}$  row of $I_{k}$
for $1\leq i \leq k$. %%
\begin{problem}
\label{sec:ProblemFormulation} {\it \textbf{LPPIEP:}}
 {\rm {Given two positive integers $k$ and $n$, a set of
partial eigenpairs $(\lambda_j, \phi_j)_{j=1}^{m} $ (where $1\leq
m\leq kn$),  construct a
  monic  matrix polynomial  $ P(\lambda)=\lambda^k I_n
+\sum_{i=0}^{k-1} \lambda^{i}
 A_{i} \in{\mathbb{R}^{n\times n}[\lambda]}$ of degree $k $ in such a way that matrices  $A_i\in \mathcal{L} $ are symmetric  for
$i=0,1,2, \ldots ,(k-1)$ and $P(\lambda)$ has the specified set
$(\lambda_j, \phi_j)_{j=1}^{m} $
 as its
eigenpairs.}}
\end{problem}
%%%%%%%%%%%%%%%%%%%%%%%%%%%%%%%%%%%%%%

\subsection*{Contributions}
In this paper, we consider the \emph{linearly structured
partial polynomial inverse eigenvalue problem for the monic
matrix polynomial of arbitrary degree $k$}. The authors believe
that this problem, in its full generality, has not been addressed
earlier in the literature. Our results solve some open problems in
the theory of polynomial inverse eigenvalue problem (see
 \cite{de2015matrix}).

In particular, key  contributions made in this paper are listed
below:
\begin{itemize}
\item
The proposed   method is capable to solve LPPIEP using
 a set  of $m$ ( $1\leq m \leq kn$) eigenpairs
without imposing any restrictions  on it, unlike some instances in
the past where certain restrictions on $m$ are imposed (see
\cite{bai2007symmetric,cai2009solutions,yuan2011class}) for
computing the solution of inverse eigenvalue problem in the case
of quadratic matrix polynomial.
\item
The proposed   method is capable to solve LPPIEP  for a monic
matrix polynomial of  arbitrary  degree $k$.
\item
We derive some necessary and sufficient conditions on the
eigendata for the existence of solution of this problem.
\item
We completely characterize the class of solutions of this problem and present the explicit expression of the solution.% Also, we
\end{itemize}

\subsection*{Real-Form Representations of Eigenvalues and Eigenvectors}

We assume that the  $m$ eigenvalues of a matrix polynomial are
given of which $t$ are complex conjugate pairs and remaining
$m-2t$ are real. Also, complex  eigenvalues are $\alpha_j \pm i
\beta_j$ for $j=1,2, \ldots ,t$ and  real eigenvalues are
$e_{2t+1},e_{2t+2},\dots, e_{m}$. Eigenvectors corresponding  to
the complex eigenvalues are $u_j \pm i v_j$ and eigenvectors
corresponding  to the real eigenvalues are $\phi_{2t+1},
\phi_{2t+2}, \ldots  \phi_{m}$.

We  relate this pair of complex eigenvalues with a matrix $E_j \in
\mathbb{R}^{2 \times 2}$ given by
\begin{equation*}
E_j= \left[
\begin{matrix}
\alpha_j& \beta_j\\
-\beta_j & \alpha_j\\
\end{matrix}
\right].
\end{equation*}
Thus given a set of $m$ eigenvalues, we relate these numbers with
a real block-diagonal matrix $E \in \mathbb{R}^{m \times m}$ of
the following form
\begin{equation}
\label{eqn:eigenvalue-vector242}
 E=diag(E_1,E_2, E_3,\dots, E_t, e_{2t+1},\dots, e_{m}).
\end{equation}
Then $E$ is  the real-form matrix representation of these $  m$
eigenvalues in real form. Similarly, for a set of $m$ eigenvectors
a real-form matrix representation is given by
\begin{equation}
\label{eqn:eigenvalue-vector243}
  X= \left[
\begin{matrix}
u_1 & v_1 & \ldots & u_t & v_t & \phi_{2t+1} & \ldots & \phi_{m}
\end{matrix}
\right]  \in \mathbb{R}^{n \times m}.
\end{equation}
 Thus the  pair $(X,E)$ is a real matrix eigenpair of the
 matrix polynomial  of degree $k$,  then it satisfies
\begin{equation}
\sum_{i=0}^{k}   A_i X
 E^i=0.
\end{equation}
This relation is known as eigenvalue eigenvector relation for the
 matrix polynomial of degree $k$.

\subsection{Linearly structured  matrices and its  structure specifications}
Linearly structured  matrix is a linear combinations of sub
structured matrices. Let $A \in \mathcal{L}$  be a linearly
structured matrix of the form
\begin{equation}
\label{eqn:define_A} A=\sum\limits_{\ell=1}^r S_\ell \alpha_\ell
\end{equation}
where $\alpha_1, \alpha_2, \dots  \alpha_r$ are the structure
parameters,  $r$ is the dimension  and \{$S_\ell$$\in$
${\mathbb{R}^{n\times n}}$  : $ \ell=1,2, \dots r$\} is a standard
basis  of $\mathcal{L}$.
 Here $\left[
\begin{smallmatrix}
\alpha_1  & \alpha_2  & \alpha_3  & \alpha_4  & \cdots   &
\alpha_{r-1} &
 \alpha_{r}
\end{smallmatrix}
\right]^T$ is the coordinate  vector of $A$ w.r.t the above
standard basis.

Matrix $A$ is the linear combinations of the sub structured
matrices $S_\ell$ for $ \ell=1,2, \dots r$.

We give some examples of linearly structured matrices in the table given below.%\pagebreak
\pagebreak
\begin{table}[h]
\caption{Linearly structured matrices }
\begin{tabular}{|p{5cm}|p{6cm}|}
%\hline
%\multicolumn{3}{|c|}{Country List} \\
\hline
Linearly structured matrix  & Dimension  $r$ \\
\hline
Symmetric  &  $\frac{n(n+1)}{2}$\\

Skew symmetric  &  $\frac{n(n-1)}{2}$\\
 Tridiagonal  & $3n-2$\\
Symmetric tridiagonal &  $2n-1$\\
Pentagonal  & $5n-6$\\
Hankel &  $2n-1$\\
Toeplitz &  $2n-1$\\
\hline
\end{tabular}
\end{table}

\section{Solution  of  LPPIEP}
\label{sec:existence_tri}
 In this section, we  obtain the solution   of
 LPPIEP from the eigenvalue-eigenvector relation for monic matrix
polynomial of degree $k$ which is  given by
\begin{equation}
\label{eqn:monic_quadratic2} \sum_{i=0}^{k-1}   A_i X E^i= -X E^k
\end{equation}
where $X\in$ ${\mathbb{R}^{n\times m}}$  and  $E\in$
${\mathbb{R}^{m\times m}}$.

It is clear that \eqref{eqn:monic_quadratic2} is a nonhomogenous
linear system of $nm$ equations.  Therefore, the solution of
LPPIEP is obtained by computing the linearly structured  solution
$A_{i}$ of \eqref{eqn:monic_quadratic2}.

We now discuss an important concept of vectorization  of a matrix
which will be used to derive  the solution of LPPIEP.

\subsection*{\textbf{Vectorization  of a
linearly structured matrix}} \label{sec:vectorization1}

Vectorization of a matrix $A\in \mathcal{L}$, is denoted by
$\mathrm{ Vec}(A)$  and  is defined as a vector in
${\mathbb{R}^{n^2\times 1}}$ obtained by stacking the columns of
the matrix $A$ on top of one another.

Define the vector $\mathrm{ Vec_1}(A)$ as
\begin{equation*}
\mathrm{ Vec_1}(A)= \left[
\begin{smallmatrix}
\alpha_1  & \alpha_2  & \alpha_3  & \alpha_4  & \cdots   &
\alpha_{r-1} &
 \alpha_{r}
\end{smallmatrix}
\right]^T.
\end{equation*}

We define the  matrix  $P \in$ ${\mathbb{R}^{n^2 \times r}}$
  as
\begin{eqnarray}
\label{eqn:define_P}
P&=&[\mathrm{Vec}(S_1)~\mathrm{Vec}(S_2)~\cdots~\mathrm{Vec}(S_r)]
\end{eqnarray}
where \{$S_\ell\in$ ${\mathbb{R}^{n\times n}}$ : $ \ell=1,2,
\ldots ,r$\} is a standard basis of $\mathcal{L}$ such that
$\mathrm{ Vec_1}(A)$  is the coordinate vector of $A\in
\mathcal{L}$ w.r.t the above  basis.

It is easy to  see   that  $\mathrm{Vec}(A)$ and
$\mathrm{Vec_1}(A)$ are related through the matrix $P$ as:
\begin{equation}
\label{eqn:monic_quadratic17*} \mathrm{ Vec}(A)= P
\hspace{.17cm}\mathrm{ Vec_1}(A)
\end{equation}

%%%
\begin{example}
Consider the symmetric  matrix(linearly structured)  $A\in$
${\mathbb{R}^{3\times 3}}$ as
\begin{equation*}
A= \left[
\begin{matrix}
4 & 2 &  8    \\
2 & 7 & 9\\
8  & 9 & 5\\
\end{matrix}
\right]
\end{equation*}
Then $\mathrm{Vec}(A)\in$ ${\mathbb{R}^{9\times 1}}$ and
$\mathrm{Vec_1}(A)\in$ ${\mathbb{R}^{5\times 1}}$   are  given by
\begin{eqnarray*}
&& \mathrm{Vec}(A) = \left[
\begin{matrix}
4   & 2   &
 8    &
 2    &
7  & 9   &
 8    &
9  & 5  &
\end{matrix}
\right]^T \\
&& \mathrm{Vec_1}(A) = \left[
\begin{matrix}
4   & 2   & 8  &
 7  &
9   & 5  &
\end{matrix}
\right]^T
\end{eqnarray*}
Let, \{$S_\ell\in$ ${\mathbb{R}^{n\times n}}$ : $ \ell=1,2, \ldots
,5$\} be the standard basis of the space of all symmetric  matrices where\\

$S_1=\left[
\begin{smallmatrix}
1 & 0 &  0    \\
0 & 0 & 0\\
0  & 0 & 0\\
\end{smallmatrix}
\right]$, $S_2=\left[
\begin{smallmatrix}
0 & 1 &  0    \\
1 & 0 & 0\\
0  & 0 & 0\\
\end{smallmatrix}
\right]$, $S_3=\left[
\begin{smallmatrix}
0 & 0 &  0    \\
0 & 1 & 0\\
0  & 0 & 0\\
\end{smallmatrix}
\right]$,

$S_4=\left[
\begin{smallmatrix}
0 & 0 &  1\\
0 & 0 & 0\\
1 & 0 & 0\\
\end{smallmatrix}
\right]$,
 $S_5=\left[
\begin{smallmatrix}
0 & 0 &  0    \\
0 & 0 & 1\\
0  & 1 & 0\\
\end{smallmatrix}
\right]$, $S_6=\left[
\begin{smallmatrix}
0 & 0 &  0    \\
0 & 0 & 0\\
0  & 0 & 1\\
\end{smallmatrix}
\right]$.

The matrix  $P \in$ ${\mathbb{R}^{9 \times 6}}$ is given by\\
%\begin{equation*}
$P= \left[
\begin{smallmatrix}
1 & 0 &   0 &  0  &  0 & 0     \\
0 & 1 &   0 &  0  &  0 & 0     \\
0 & 0 &   0 &  1  &  0 & 0     \\
0 & 1 &   0 &  0  &  0 & 0    \\
0 & 0 &   1 &  0  &  0 & 0    \\
0 & 0 &   0 &  0   & 1 & 0     \\
0 & 0 &   0 &  1   & 0 & 0    \\
0 & 0 &   0 &  0   & 1 & 0    \\
0 & 0 &   0 &  0   &  0 & 1    \\
\end{smallmatrix}
\right]$

For the symmetric  matrix $A$, it is  straightforward to verify
that \eqref{eqn:monic_quadratic17*} holds.
\end{example}
%%%%%%%%%%%%%%%%%%
\subsection*{\textbf{Existence of a  solution  of  LPPIEP}}
\label{sec:existence_affine} In this subsection, we derive  a
necessary and sufficient condition on the eigendata for the
existence of a solution of LPPIEP. Applying vectorization
operation on \eqref{eqn:monic_quadratic2}, we get,
\begin{eqnarray}
&&\label{eqn:monic_quadratic270}
\mathrm{ Vec}\left(\sum_{i=0}^{k-1}   A_i X E^i \right)=-\mathrm{Vec}\left( X E^k\right)\nonumber\\
&&\Rightarrow\label{eqn:monic_quadratic28}\sum_{i=0}^{k-1} \left((
X
E^{i})^T\otimes I\right) \mathrm{Vec}\left(A_{i}\right)=-\mathrm{Vec}\left( X E^k\right)\nonumber\\
&&\Rightarrow\label{eqn:monic_quadratic29}
 \sum_{i=0}^{k-1}  \left( ( X
E^{i})^T\otimes I \right)P \mathrm{Vec_1}\left(A_{i}\right)=-\mathrm{Vec}\left( X E^k \right)~~\ldots\textrm{using \eqref{eqn:monic_quadratic17*}} \nonumber\\
&& \Rightarrow \label{eqn:monic_quadratic311}
  \left[ \begin{smallmatrix} ((X E
^{k-1})^T\otimes I)P  & (( X E^{k-2})^T\otimes I)P & \cdots &
 (X^T\otimes I) P
\end{smallmatrix}
\right] \left[
\begin{smallmatrix}
\mathrm{Vec_1}( A_{k-1})\\
\mathrm{Vec_1}( A_{k-2})\\
\vdots\\
 \mathrm{Vec_1}( A_0)
\end{smallmatrix}
\right]
 =\mathrm{-Vec}( X E^k)\nonumber\\
  && \Rightarrow \label{eqn:monic_quadratic9}
 U x=b
\end{eqnarray}
where
\begin{eqnarray}
\label{eqn:define_U} && U=\left[ \begin{smallmatrix} ((X E
^{k-1})^T\otimes I_n)P  & (( X E^{k-2})^T\otimes I_n)P & \cdots &
 (X^T\otimes I_n) P
\end{smallmatrix}
\right]\in \mathbb{R}^{mn \times {k}r}, \\
\label{eqn:define_x} && x=  \left[
\begin{smallmatrix}
\mathrm{Vec_1}( A_{k-1})\\
\mathrm{Vec_1}( A_{k-2})\\
\vdots\\
 \mathrm{Vec_1}( A_0)
\end{smallmatrix}
\right]\in \mathbb{R}^{{kr} \times 1},\\
\label{eqn:define_b} && b=\mathrm{ Vec}(- XE^k) \in \mathbb{R}^{mn
\times 1}.
\end{eqnarray}
Above  system of linear equations \eqref{eqn:monic_quadratic9} has
$mn$ equations and $kr$ unknowns. We now state a necessary and
sufficient condition for the existence of the solution of a system
of linear equations in the following theorem.
%%%%%%%%%%%%%%%%%%%%%%%%%%%%%%%%%%%%%%%%%%%%%%%%%%%%%%%%%%%%%%%
\begin{theorem}
\label{eqn:monic_quadratic38}\cite{ben2003generalized} Let $\Psi
\zeta = \eta$ be a system of linear equations where  $\Psi \in
\mathbb{R}^{p \times q}$ and $\eta \in \mathbb{R}^{p}$. Then $\Psi
\zeta = \eta$ is consistent if and only if $\Psi \Psi^\dag
\eta=\eta$ where $\Psi^\dag$ is the generalized inverse of $\Psi
\in \mathbb{R}^{p \times q}$. %and  $b\in \mathbb{R}^{m^2 \times 1}$.
General solution of $\Psi \zeta = \eta$ is given by
\begin{equation*}
\label{eqn:monic_quadratic33} \zeta=\Psi^\dag \eta+(I_{q}
-\Psi^\dag \Psi) y
\end{equation*}
where $y \in \mathbb{R}^{q \times 1}$ is an arbitrary vector.
Moreover, $\Psi \zeta = \eta$  has a unique solution if and only
if  $\Psi^\dag \Psi=I_{q}$, $\Psi \Psi^\dag \eta=\eta$  and the
unique solution is given by
\begin{equation*}
\label{eqn:monic_quadratic34} \zeta=\Psi^\dag \eta
\end{equation*}
\end{theorem}
 First, we transform  the eigenvalue eigenvector relation
\eqref{eqn:monic_quadratic2} to a system of linear equations
$Ux=v$. Therefore, determination of
   solution  of LPPIEP
   is equivalent to finding  the   solution
   of the system of linear  equations in \eqref{eqn:monic_quadratic9}.
Thus, necessary and sufficient conditions for the existence of
solution of LPPIEP is same as the system of linear equation
$Ux=v$. We now present  the  main theorem to find a necessary and
sufficient condition for the existence of the solution of LPPIEP.

\begin{theorem}
\label{eqn:monic_quadratic44} Let an  arbitrary  matrix eigenpair
$(E,X)\in \mathbb{R}^{m \times m} \times \mathbb{R}^{n \times m}$
be given as  in Equations \eqref{eqn:eigenvalue-vector242} and
\eqref{eqn:eigenvalue-vector243}. Then  LPPIEP has a solution
 if
and only if $U U^\dag b=b$  where $U$ and $b$ are  defined by
\eqref{eqn:define_U} and \eqref{eqn:define_b}. In that case
expression of $A_{i}\in\mathcal{L}$ for $i=0,1,2,\ldots ,(k-1)$
are given by
\begin{eqnarray}
\label{eqn:Expresion_A}
 && \mathrm{Vec}(A_{i}) = P    \left( (e_{k-i} \otimes
I_{r})\left(U^\dag b+(I_{kr} -U^\dag U) y\right) \right),
\end{eqnarray}
where $y \in \mathbb{R}^{kr \times 1}$ is an arbitrary vector.
Moreover, LPPIEP  has a unique solution  if and only if $U U^\dag
b=b$, $U^\dag U=I_{kr}$. Explicit expressions of
$A_{i}\in\mathcal{L}$
  are given below as
\begin{eqnarray}
&& \mathrm{Vec}(A_{i}) = P    \left( (e_{k-i} \otimes I_{r}
)U^\dag b \right).
\end{eqnarray}
\end{theorem}
%%%%%%%%%%%%%%%%%%
\begin{proof}
Computing the  solution of LPPIEP  is  equivalent to solving the
 system of linear equations $Ux=b$ where $U$ and $b$
are defined by \eqref{eqn:define_U} and \eqref{eqn:define_b}.
Necessary and sufficient condition for the existence of the
 solution of  $Ux=b$ is $U U^\dag
b=b$ and  general solution  is given by
\begin{equation}
\label{eqn:monic_quadratic334455} x=U^\dag b+(I_{kr} -U^\dag U) y.
\end{equation}
where $y \in \mathbb{R}^{kr  \times 1}$ is an arbitrary vector.
Note that, $x$ is of the form as in \eqref{eqn:define_x} and
$\mathrm{Vec_1}( A_{k-1})$ can be obtained from $x$ as follows.
\begin{equation*}
\left[
\begin{matrix}
 I_{r} & \Theta  & \Theta  &  \Theta  & \hdots \Theta
\end{matrix}
\right]x= \left[
\begin{matrix}
 I_{r} & \Theta  & \Theta  &  \Theta  & \hdots \Theta
\end{matrix}
\right]\left[
\begin{matrix}
\mathrm{Vec_1}( A_{k-1})\\
\mathrm{Vec_1}( A_{k-2})\\
\vdots\\
 \mathrm{Vec_1( A_0)}
\end{matrix}
\right]=\mathrm{Vec_1( A_{k-1})}
\end{equation*}
where $\Theta$ $\in \mathbb{R}^{r \times r}$ be a zero matrix.
%$\Rightarrow$
\begin{eqnarray}
\label{eqn:eigenvalue-vector356}
 &&\Rightarrow \mathrm{Vec_1}( A_{k-1})=(e_{1}
\otimes I_{r} )x
\end{eqnarray}
%%%%%%%%%%%%%%%%%%
%%%%%%%%%%%%%%%%%%%%%%%%%%%%%%%%%%%%%%%%%%%%%%%%%%%%%%%%%%%%
Similarly, $\mathrm{Vec_1}( A_{i})$  are given by
\begin{eqnarray}
\label{eqn:eigenvalue-vector3582}
 &&\mathrm{Vec_1}( A_{i})=\big(e_{k-i}
\otimes I_{r}\big )x ~~ \textrm{for  } i=0,1,2 ,3, \ldots,(k-1)
\end{eqnarray}
%%%%%%%%%%%%%55
Substituting  the expression  of $x$ in
\eqref{eqn:eigenvalue-vector3582}, $\mathrm{Vec_1}(A_{i})$
 can be obtained as in the following as :
\begin{eqnarray}
\label{eqn:eigenvalue-vector35822}
 &&\mathrm{Vec_1}( A_{i})=\left(e_{k-i}
\otimes I_{r}\right )\left(U^\dag b+(I_{kr} -U^\dag U) y\right)
\end{eqnarray}
General solution $A_{i}$  is obtained from the vector
$\mathrm{Vec_1}( A_{i})$  using
 the  relation \eqref{eqn:monic_quadratic17*} as
\begin{eqnarray}
&& \mathrm{Vec}(A_{i}) = P \mathrm{Vec_1}( A_{i}).
\end{eqnarray}
Substituting the expressions  of $\mathrm{Vec_1}(A_{i})$  in the
above equations, we get
\begin{eqnarray}
&& \mathrm{Vec}(A_{i}) = P    \left( (e_{k-i} \otimes
I_{r})\left(U^\dag b+(I_{kr} -U^\dag U) y\right) \right).
\nonumber
\end{eqnarray}
 Further, $Ux=b$ has a
unique solution if and only if $U U^\dag b=b$ and $U^\dag
U=I_{kr}$.  Explicitly, the unique solution $x$ is  given by $
x=U^\dag b$ (see Theorem \ref{eqn:monic_quadratic38}). If $Ux=b$
has a unique solution  then LPPIEP  has a unique solution $A_{i}$.
In that case, matrices $A_{i}$ are given by uniquely as
\begin{eqnarray}
&&\mathrm{Vec}(A_{i}) = P    \left( (e_{k-i} \otimes I_{r} )U^\dag
b \right).
\end{eqnarray}
\end{proof}
\begin{remark}
We considered  the standard ordered basis of $\mathcal{L}$ to
represent any linearly structured matrix and
 we construct the matrix $P$ using this basis. However,  any other
 ordered basis can be chosen  to construct the matrix
$P$. Result of Theorem \eqref{eqn:monic_quadratic44} is also true
if we choose any  other  ordered basis.
\end{remark}

%%%%%%%%%%%%%%%%%%%%%%%%%%%%%%%%%%%%%%%%%%%%%%%%%%%%%%%%%%%%%%%%
\subsection*{\textbf{Construction of symmetric non-monic matrix polynomials  }}
 In Theorem
\ref{eqn:monic_quadratic44}, we construct the monic linearly
structured matrix polynomial using partial eigendata. Now we
generalize this result to find the symmetric non-monic polynomials
with positive definite leading coefficients using similarity
transformation.

Consider the matrix polynomial $P(\lambda)= \lambda^k A_k +
\lambda^{k-1} A_{k-1}+\cdots+\lambda
A_1+A_0\in{\mathbb{R}^{n\times n}[\lambda]}$ where $A_{i}$  are
symmetric and $A_k $ is positive definite matrix. Let, $A_k^{1/2}$
be the positive definite square-root of $A_k$  and modify the
problem by writing $\xi =A_k^{1/2}x$ and observe that Eq.
\eqref{eqn:system_polyEIG} reduces to the monic problem as
\begin{eqnarray}
&&(\lambda^k A_k^{-1/2} A_k A_k^{-1/2} + \lambda^{k-1} A_k^{-1/2}
A_{k-1} A_k^{-1/2} +\cdots
+A_k^{-1/2} A_0 A_k^{-1/2})A_k^{1/2}x=0\nonumber\\
&&\Rightarrow (\lambda^k I + \lambda^{k-1} \hat{ A}_{k-1}
 +\cdots+\lambda  \hat{ A}_1 +
\hat{ A}_0 )\xi=0\nonumber
\end{eqnarray}
where $\hat{ A}_i=A_k^{-1/2} A_{i} A_k^{-1/2}$ are   symmetric
matrices.

%%%%%%%%%%%%%%%%%%%%%%%%%%%%%%%%%%%%%%%%%%%%%%%%%%%%%%%%%%%%%%%%%%%%%%%%%%%%%%%%%%%%%%%%%%%%%%%%%%%%%%%%%%%%%%
\section{Numerical Example}
In this section, we give three numerical examples to illustrate
the validity of  our proposed approach.

%%%%%%%%%%%%%%%%%%%%%%%%%%%%%%%%%%%%%%%%%%%%%%%%%%%%%%%%%%%%%%5
\begin{example}
\label{eqn:examples2}
 {\rm
%%%%%% 0.078
Consider the mass-spring system  having three degrees of freedom
with the following target set of eigenvalues  $  -1.3064\pm
0.5436i$, $-0.2582$.
%%%%%%
 The eigenvalue  and the  eigenvector
 matrices are given by
\begin{equation*}
X= \left[
\begin{smallmatrix}
    -0.0406  &  -0.4699  &  0.4231\\
    -0.4504  &  -0.2542 &   0.3510\\
    0.7128   &  -0.0438 &  -0.8353\\
\end{smallmatrix}
 \right]
\end{equation*}
%We define the eigenvalue matrix $T$ as
\begin{equation*}
E= \left[
\begin{smallmatrix}
  -1.3064      &     0.5436     &    0\\
    -0.5436    &    -1.3064     &     0 \\
        0      &      0         &   -0.2582\\
\end{smallmatrix}
 \right]
\end{equation*}

 For this mass-spring system $m=3$ and $n=3$.
Now we construct  the monic  symmetric matrix  polynomial
$P(\lambda)=  \lambda^{2}I_3 +
 \lambda A_1+
 A_0$   of
degree $2$.

We take the standard basis of the space of all symmetric matrices
\{$S_\ell\in$ ${\mathbb{R}^{n\times n}}$ : $ \ell=1,2, \ldots
,6$\}  where $S_1=\left[
\begin{smallmatrix}
1 & 0 &  0    \\
0 & 0 & 0\\
0  & 0 & 0\\
\end{smallmatrix}
\right]$, $S_2=\left[
\begin{smallmatrix}
0 & 1 &  0    \\
1 & 0 & 0\\
0  & 0 & 0\\
\end{smallmatrix}
\right]$, $S_3=\left[
\begin{smallmatrix}
0 & 0 &  1    \\
0 & 0 & 0\\
1  & 0 & 0\\
\end{smallmatrix}
\right]$,
 $S_4=\left[
\begin{smallmatrix}
0 & 0 &  0    \\
0 & 1 & 0\\
0  & 0 & 0\\
\end{smallmatrix}
\right]$, $S_5=\left[
\begin{smallmatrix}
0 & 0 &  0    \\
0 & 0 & 1\\
0  & 1 & 0\\
\end{smallmatrix}
\right]$,$S_6=\left[
\begin{smallmatrix}
0 & 0 &  0    \\
0 & 0 & 0\\
0  & 0 & 1\\
\end{smallmatrix}
\right]$.

Now, we construct  symmetric   matrices $A_0$ and $A_1$ from the
above partial eigendata.
%%%%%
Here,  $UU^\dag b=b$ and $U^\dag U \neq I_{10}$  where $U
\in{\mathbb{R}^{12 \times 10}}$ and $b \in{\mathbb{R}^{12 \times
1}}$. Equation \eqref{eqn:monic_quadratic9} has an infinite number
of solutions. Therefore, LPPIEP has an infinite number of
solutions. Using Theorem  \ref{eqn:monic_quadratic44},  symmetric
 matrices $A_0$ and $A_1$  are given by
%%%%%
   \begin{equation*}
A_0= \left[
\begin{matrix}
     4.2248  & -0.0174  &  2.4278\\
   -0.0174  &  1.8133  &  0.2806\\
    2.4278  &  0.2806  &  1.5618\\
\end{matrix}
 \right]
\end{equation*}
%%%%%
\begin{equation*}
A_1= \left[
\begin{matrix}
    2.3283  &  1.2405 &   2.7130\\
    1.2405  &  0.1189 &  -1.2603\\
    2.7130 &  -1.2603 &   1.9321\\
\end{matrix}
 \right]
\end{equation*}
%%%%%
Next, we study the effect of choosing different ordered basis of
the space of all symmetric matrices to the solution. Using the
ordered basis \{$S_\ell\in$ ${\mathbb{R}^{n\times n}}$ : $
\ell=1,2, \ldots ,6$\}  where $S_1=\left[
\begin{smallmatrix}
1 & 0 &  0    \\
0 & 0 & 0\\
0  & 0 & 0\\
\end{smallmatrix}
\right]$, $S_2=\left[
\begin{smallmatrix}
0 & 1 &  0    \\
1 & 0 & 0\\
0  & 0 & 0\\
\end{smallmatrix}
\right]$, $S_3=\left[
\begin{smallmatrix}
1 & 0 &  1    \\
0 & 0 & 0\\
1  & 0 & 0\\
\end{smallmatrix}
\right]$,
 $S_4=\left[
\begin{smallmatrix}
0 & 0 &  0    \\
0 & 1 & 0\\
0  & 0 & 0\\
\end{smallmatrix}
\right]$, $S_5=\left[
\begin{smallmatrix}
0 & 0 &  0    \\
0 & 0 & 1\\
0  & 1 & 0\\
\end{smallmatrix}
\right]$,$S_6=\left[
\begin{smallmatrix}
0 & 0 &  0    \\
0 & 0 & 1\\
0  & 1 & 1\\
\end{smallmatrix}
\right]$, we construct the matrices $P$ and $D$. Here,  $UU^\dag
b=b$ and $U^\dag U \neq I_{10}$  where $U \in{\mathbb{R}^{12
\times 10}}$ and $b \in{\mathbb{R}^{12 \times 1}}$.
 We also  get the same symmetric
 matrices $A_0$ and $A_1$ as above.

Therefore, if we take two different basis of the space of all
symmetric matrices  for this example, we get  same result.}
\end{example}

\begin{example}
\label{eqn:examples1} {\rm
%%%%%% 0.078
Consider the mass-spring system  having four degrees of freedom
with the following target set of eigenvalues  $  0.5950 +
9.5092i$, $0.5950 - 9.5092i$.
%%%%%%
 The eigenvalue  and the  eigenvector
 matrices are given by
\begin{equation*}
X= \left[
\begin{matrix}
 -0.2164 &    -0.6066\\
  -0.5435  &     -0.0169\\
  -0.3518  &      0.2746\\
  -0.1845  &      0.2374\\
\end{matrix}
 \right]
\end{equation*}
%We define the eigenvalue matrix $T$ as
\begin{equation*}
E= \left[
\begin{matrix}
 0.5950 &  9.5092\\
    -9.5092 &   0.5950\\
\end{matrix}
 \right]
\end{equation*}
 For this mass-spring system $m=2$ and $n=4$.
%%%%%

Now we construct  the monic skew symmetric matrix  polynomial
$P(\lambda)=  \lambda^{2}I_4 +
 \lambda A_1+
 A_0$   of
degree $2$.

We take the standard basis of the space of all skew symmetric
matrices \{$S_\ell\in$ ${\mathbb{R}^{n\times n}}$ : $ \ell=1,2,
\ldots ,6$\}  where $S_1=\left[
\begin{smallmatrix}
0 & 1 &  0  &  0  \\
-1 & 0 & 0   &  0 \\
0  & 0 & 0  &  0 \\
0  & 0 & 0  &  0 \\
\end{smallmatrix}
\right]$, $S_2=\left[
\begin{smallmatrix}
0  & 0 &  1   &  0  \\
0  & 0 &  0    &  0 \\
-1 & 0 &  0   &  0 \\
0  & 0 &  0   &  0 \\
\end{smallmatrix}
\right]$, $S_3=\left[
\begin{smallmatrix}
0  & 0 &  0   &  1  \\
0  & 0 &  0   &  0 \\
0  & 0 &  0   &  0 \\
-1 & 0 &  0   &  0 \\
\end{smallmatrix}
\right]$, $S_4=\left[
\begin{smallmatrix}
0  & 0 &  0   &  0  \\
0  & 0 &  1  &  0 \\
0  & -1 &  0   &  0 \\
0 & 0 &  0   &  0 \\
\end{smallmatrix}
\right]$, $S_5=\left[
\begin{smallmatrix}
0  & 0 &  0   &  0  \\
0  & 0 &  0   &  1 \\
0  & 0 &  0   &  0 \\
0 & -1 &  0   &  0 \\
\end{smallmatrix}
\right]$, $S_6=\left[
\begin{smallmatrix}
0  & 0 &  0   &  0 \\
0  & 0 &  0   &  0 \\
0  & 0 &  0   &  1 \\
0 & 0 &  -1  &  0 \\
\end{smallmatrix}
\right]$.

 Here,  $UU^\dag b=b$ and $U^\dag U \neq I_{12}$  where
$U \in{\mathbb{R}^{8 \times 12}}$ and $b \in{\mathbb{R}^{8 \times
1}}$. Equation \eqref{eqn:monic_quadratic9} has an infinite number
of solutions. Therefore, LPPIEP has an infinite number of
solutions. One of the solution $x$ of Equation
\eqref{eqn:monic_quadratic9} is given by\\
%%%%%
%\begin{eqnarray*}
$x=[ 6.1761\hspace{.15 cm}
    5.1682\hspace{.15 cm}
    3.0933\hspace{.15 cm}
    2.9398\hspace{.15 cm}
    2.5033\hspace{.15 cm}
    0.6224\hspace{.15 cm}
    3.7036\hspace{.15 cm}
    3.0992\hspace{.15 cm}
    1.8550\hspace{.15 cm}\\
    1.7629\hspace{.15 cm}
    1.5011\hspace{.2 cm}
    0.3732].^T$

Using Theorem  \ref{eqn:monic_quadratic44} matrices $A_0$ and
$A_1$  are given by
%%%%%
   \begin{equation*}
A_0= \left[
\begin{smallmatrix}
        0 &    3.7036  &   3.0992   &  1.8550\\
   -3.7036   &       0  &   1.7629 &    1.5011\\
   -3.0992 &   -1.7629  &        0  &   0.3732\\
   -1.8550  &  -1.5011 &   -0.3732  &        0\\
\end{smallmatrix}
 \right]
\end{equation*}
%%%%%
\begin{equation*}
A_1= \left[
\begin{smallmatrix}
        0  &   6.1761  &   5.1682 &    3.0933\\
   -6.1761  &        0 &    2.9398  &   2.5033\\
   -5.1682 &   -2.9398  &        0  &   0.6224\\
   -3.0933 &   -2.5033 &   -0.6224  &        0\\
\end{smallmatrix}
 \right]
\end{equation*}
%%%%%
Constructed matrices $A_0$ and $A_1$ are skew symmetric and they
satisfy the eigenvalue and eigenvector relation $X E ^2+A_1 X
E+A_0 X=0$ as
 $\Vert X E^2+ A_1 X E+A_0 X  \Vert _F^2$= $ 8.0185   \times 10^{-6}$.
Total computational time for running this program in a system with
$4$Gb ram is $0.078$ seconds.
%%%%%%%%%%%
Therefore, we  successfully reproduced the  eigenvalues and
eigenvectors from the constructed monic skew symmetric quadratic
matrix polynomial.

Next, we study the effect of choosing different ordered basis of
the space of all skew symmetric matrices to the solution.

Using the ordered basis \{$S_\ell\in$ ${\mathbb{R}^{n\times n}}$ :
$ \ell=1,2, \ldots ,6$\} where $S_1=\left[
\begin{smallmatrix}
0 & 1 &  -2  &  0  \\
-1 & 0 & 0   &  0 \\
2  & 0 & 0  &  0 \\
0  & 0 & 0  &  0 \\
\end{smallmatrix}
\right]$, $S_2=\left[
\begin{smallmatrix}
0  & 0 &  1   &  0  \\
0  & 0 &  0    &  0 \\
-1 & 0 &  0   &  0 \\
0  & 0 &  0   &  0 \\
\end{smallmatrix}
\right]$, $S_3=\left[
\begin{smallmatrix}
0  & 0 &  0   &  1  \\
0  & 0 &  0   &  0 \\
0  & 0 &  0   &  0 \\
-1 & 0 &  0   &  0 \\
\end{smallmatrix}
\right]$, $S_4=\left[
\begin{smallmatrix}
0  & 0 &  0   &  0  \\
0  & 0 &  1  &  0 \\
0  & -1 &  0   &  0 \\
0 & 0 &  0   &  0 \\
\end{smallmatrix}
\right]$, $S_5=\left[
\begin{smallmatrix}
0  & 0 &  0   &  0  \\
0  & 0 &  0   &  1 \\
0  & 0 &  0   &  0 \\
0 & -1 &  0   &  0 \\
\end{smallmatrix}
\right]$, $S_6=\left[
\begin{smallmatrix}
0  & 0 &  0   &  0 \\
0  & 0 &  0   &  0 \\
0  & 0 &  0   &  1 \\
0 & 0 &  -1  &  0 \\
\end{smallmatrix}
\right]$, matrices $A_0$ and $A_1$  are given by
%%%%%
   \begin{equation*}
A_0= \left[
\begin{smallmatrix}
        0 &    -1.2396  &   6.4982   &  2.0008\\
   1.2396   &       0  &   4.0440 &    3.6581\\
   -6.4982 &   -4.0440  &        0  &   0.3732\\
   -2.0008  &  -3.6581 &   -0.3732  &        0\\
\end{smallmatrix}
 \right]
\end{equation*}
%%%%%
\begin{equation*}
A_1= \left[
\begin{smallmatrix}
        0  &   6.1815  &   5.1892 &    3.6862\\
   -6.1815  &        0 &    2.7181  &   1.7956 \\
   -5.1892 &   -2.7181  &        0  &   1.3404\\
   -3.6862 &   -1.7956  &   -1.3404  &        0\\
\end{smallmatrix}
 \right]
\end{equation*}
If we take two different basis of the space of all skew symmetric
matrices  for this example, we get  different skew symmetric
matrices. }
\end{example}

\begin{example}
{\rm Consider a $50 \times 50$ triplet $(I_{50}, A_1,A_0)$ where
symmetric tridiagonal matrices $A_0,A_1$ are  generated using the
MATLAB  as
\begin{eqnarray*}
A_1 &=& \mathrm{diag}(a_1)+\mathrm{diag}(b_1,-1)+\mathrm{diag}(b_1,1) \\
A_0 &=&
\mathrm{diag}(a_2)+\mathrm{diag}(b_2,-1)+\mathrm{diag}(b_2,1)
\end{eqnarray*}
where
 $a_1$=[$10$
   $20$
    $6$
    $8$
   $40$
   $10$
  $ 50$ $ 60$
    $3$  $70$ $ 30$   $ 7$ $ 9$  $ 4$  $80$
   $ 4.2$
    $6.5$  $8.1$   $1.2$  $6.2$   $2.7$  $ 4.3$   $3.2$   $2.6$  $14$
   $ 2.9$
  $ 13$ $ 12.4$ $ 4.6$ $14.2$  $8$
    $1.9$  $2.4$ $ 1.6$ $ 25$  $10.84$ $ 22.3$  $ 42.62$ $54.24$
   $26.24$
   $ 1$
    $4$   $0.5$   $0.3$   $ 7$  $3$   $8$   $0.9$  $5$   $0.2$], \\
 $b_1$=[ $2.8$
   $ 1.2$  $36$   $8$   $4$  $16$   $2$  $ 1.2$ $ 28$  $ 12$ $  32$  $3.6$ $   20$  $ 0.8$
   $ 1.8$  $ 0.96$  $ 3.92$  $3.24$ $ 1.04$  $6$  $  0.9$  $3$ $ 0.4$
   $ 4$
   $ 0.2$
   $ 2$  $0.5$  $0.6$  $0.8$   $0.3$  $ 2$  $ 1$  $ 6$  $0.9$   $ 3$   $ 0.4$   $ 4$  $ 0.2$ $ 2$  $ 5$  $ 2$  $ 1$   $0.7$  $ 8$  $ 0.2$  $ 0.6$  $7$   $0.4$   $7$], \\
$a_2$=[ $5.6$  $ 2.4$ $ 16$  $ 8$  $48$  $ 7.2$ $ 24$  $  3.2 $ $
32 $  $ 1.6 $ $  16 $
   $ 4 $  $ 4.8 $  $  6.4 $ $ 72 $ $  80 $$  168 $ $  328$  $432$ $ 200$  $ 17.6$ $ 26.4 $ $  23.2$
  $ 17.6$
  $ 96$
  $ 19.2$
  $ 84 $
  $ 75.2 $  $ 35.6 $ $  85.6$  $52$   $12.4$   $15.6$  $ 11.2$ $168$  $85.04$ $ 175.8$ $ 337.72$ $ 433.44$ $ 207.44$   $ 0.4$  $4$   $0.2$   $2$  $ 0.5$  $ 0.6$  $ 0.8$  $  9$ $ 10 $  $ 21$], \\
$b_2$=[$  3.2$
   $ 3.6$
 $  16$
  $ 20$
  $  8$
  $  4$
   $ 2.8$
  $ 32$
   $ 0.8$ $  2.4$ $ 28$ $  1.6$  $ 28$  $  2$  $ 76$ $ 96$ $ 112$ $ 136$ $ 204$  $ 4$ $  0.2$ $  2$ $   0.5$ $   0.6$  $  0.7$ $ 0.3$  $ 2$
  $ $  1  $ $  6 $ $  8 $ $ 16 $ $  4.8 $ $  6.4 $ $ 32 $ 8  $ $ 40 $ $ 48
  $ $  2.4 $ $ 56 $ $  24 $ $  5.6 $ $  7.2  $ $  3.2  $ $ 64  $ $  3.36  $ $  5.2  $ $  6.48  $ $  0.96 $ $ 4.96$].

We compute all $100$ eigenpairs of $P(\lambda)= \lambda^{2} I_{50}
+ \lambda A_1+A_0$.   Now, we consider the eigenvalues $-1.5564
\pm 0.0232i$, $-2.5036$, $-2.1202$ and corresponding eigenvectors.
We  construct the matrix eigenpairs $(E,X)\in \mathbb{R}^{4 \times
4} \times \mathbb{R}^{50 \times 4}$ using the given four
eigenvalues  and corresponding eigenvectors. Here, $n=50$, $k=2$
and $m=4$. We construct the matrices $U\in{\mathbb{R}^{200 \times
198}}$, $b\in{\mathbb{R}^{200 \times 1}}$ and  observe that
$UU^\dag b=b$ and $U^\dag U = I_{198}$.  Therefore, LPPIEP has  a
unique solution. We construct the symmetric tridiagonal matrices
$A_0\in{\mathbb{R}^{50 \times 50}}$ and $A_1\in{\mathbb{R}^{50
\times 50}}$ using Theorem \ref{eqn:monic_quadratic44} and they
satisfy the eigenvalue and eigenvector relation $X E ^2+A_1 X
E+A_0 X=0$ as
 $\Vert X E^2+ A_1 X E+A_0 X  \Vert _F$= $ 7.1   \times 10^{-8}$.
 Total computational time for running this program in a system with
$4$ GB ram is $1.158$ seconds.
%%%%%%%%%%%
Therefore, we  successfully reproduced the  eigenvalues and
eigenvectors from the constructed monic  symmetric tridiagonal
quadratic matrix polynomial.}

{\rm Similarly for various cases of partial eigendata where
$m=2,6$ and $10$, we construct the matrix eigenpairs $(E,X)\in
\mathbb{R}^{m \times m} \times \mathbb{R}^{50 \times m}$. We
construct the symmetric tridiagonal matrices
$A_0\in{\mathbb{R}^{50 \times 50}}$ and $A_1\in{\mathbb{R}^{50
\times 50}}$ using Theorem \ref{eqn:monic_quadratic44}. The
numerical results are summarized in the following table.
\begin{table}[h]
\caption{Summary of numerical results}
\begin{tabular}{|p{.4 cm}|p{.4 cm}|p{4cm}|p{4.1 cm}| p{1.4 cm}|p{1.1 cm}|}
%\hline
%\multicolumn{3}{|c|}{Country List} \\
\hline
n    &  m   &  Conditions Satisfied        &    $\Vert X E^2+ A_1 X E+A_0 X  \Vert _F$ & Solution & Time(s)  \\
\hline
 50    & 2 &    $U U^\dag b=b$,$U^\dag U \neq I_{198}$                     &   $ 2.5   \times 10^{-11}$ &  Infinite  & $1.12$ s \\
\hline
50   &  4  &   $U U^\dag b=b$, $U^\dag U= I_{198}$                      &    $ 7.1   \times 10^{-8}$ & Unique & $1.15$ s   \\
\hline
 50 & 6   &    $U U^\dag b=b$, $U^\dag U=I_{198}$                      &  $ 3.6   \times 10^{-8}$ &  Unique  &  $1.22$ s    \\
\hline
 50 & 10   &   $U U^\dag b=b$,  $U^\dag U = I_{198}$                     &    $ 5.74   \times 10^{-6}$ & Unique   &  $1.29$ s    \\
\hline
\end{tabular}
\end{table}
%%%%%%%%%%%%%555
}
\end{example}
\section{Conclusions}
\label{sec:Conclusions} In this paper, we have studied  the
 linearly structured  partial  polynomial inverse
eigenvalue problem.  A necessary and sufficient condition for the
existence of solution to this problem is derived in this paper.
Additionally, we present an analytical expression of the solution.
Further, we discuss the sensitivity of the solution when the
eigendata is not exactly known. Thus, this paper presents a
complete theory on the structured solution of the inverse
eigenvalue problem for a monic matrix polynomial of arbitrary
degree.
%%%%%%%%%%%

\section*{References}
\bibliographystyle{plain}
\bibliography{AFFINE_BIB1}
\end{document}